\documentclass[11pt]{amsart}

\usepackage{amsmath}
\usepackage{amssymb}
\usepackage{amsthm}
\usepackage{amscd, amsfonts}
\usepackage[dvips]{epsfig}
\usepackage{verbatim}
\usepackage{epsf}

\newcommand{\bi}{\bar{i}}
\newcommand{\bj}{\bar{j}}
\newcommand{\bk}{\bar{k}}
\newcommand{\bl}{\bar{l}}

\newcommand{\bq}{\bar{q}}

\newcommand{\al}{\alpha}
\newcommand{\bal}{\bar{\alpha}}
\newcommand{\be}{\beta}
\newcommand{\bbe}{\bar{\beta}}
\newcommand{\ga}{\gamma}

\newcommand{\De}{\Delta}
\newcommand{\de}{\delta}

\newcommand{\eps}{\epsilon}
\newcommand{\ka}{\kappa}
\newcommand{\bka}{\bar{\kappa}}
\newcommand{\la}{\lambda}
\newcommand{\bla}{\bar{\lambda}}

\newcommand{\m}{\mu}
\newcommand{\bm}{\bar{\mu}}

\newcommand{\bn}{\bar{\nu}}

\newcommand{\tta}{\theta}

\newcommand{\si}{\sigma}
\newcommand{\bsi}{\bar{\sigma}}
\newcommand{\Si}{\Sigma}

\newcommand{\ta}{\tau}

\newcommand{\fD}{\mathfrak{D}}

\newcommand{\fz}{\mathfrak{z}}

\newcommand{\bfz}{\bar{\mathfrak{z}}}

\newcommand{\calA}{\mathcal{A}}

\newcommand{\calE}{\mathcal{E}}

\newcommand{\calG}{\mathcal{G}}
\newcommand{\calH}{\mathcal{H}}

\newcommand{\calO}{\mathcal{O}}

\newcommand{\calR}{\mathcal{R}}
\newcommand{\calS}{\mathcal{S}}

\newcommand{\om}{\omega}
\newcommand{\Om}{\Omega}

\newcommand{\bbC}{\mathbb{C}}

\newcommand{\bbR}{\mathbb{R}}

\newcommand{\dc}{d^c}

\newcommand{\del}{\operatorname{\partial}}
\newcommand{\bdel}{\operatorname{\bar{\partial}}}

\newcommand{\supp}{\operatorname{supp}}

\newcommand{\wge}{\wedge}

\theoremstyle{plain}
\newtheorem{theorem}{Theorem}[section]
\newtheorem{lemma}[theorem]{Lemma}
\newtheorem{proposition}[theorem]{Proposition}
\newtheorem{corollary}[theorem]{Corollary}

\theoremstyle{definition}
\newtheorem{remark}[theorem]{Remark}
\newtheorem{definition}[theorem]{Definition}

\numberwithin{equation}{section}

\newtheorem*{sub}{}

\author[Aleyasin]{S. Ali Aleyasin}
\address{Department of Mathematics\\
Stony Brook University\\ Stony Brook, NY 11794}
\email{aleyasin@math.sunysb.edu}

\author[Chen]{Xiu-Xiong Chen}
\address{Department of Mathematics\\
Stony Brook University\\ Stony Brook, NY 11794}
\email{xiu@math.sunysb.edu}

\thanks{The first author is grateful to Prof. Eric Bedford for fruitful discussions and his useful comments. The authors would like to thank Weiyong He and Zheng Kai for their careful reading of an earlier version of this note and pointing out a slight inaccuracy in the draft. }

\title[Geodesics in the space of K\"ahler metrics with prescribed singularities]{on the Geodesics 
in the space of K\"ahler metrics with prescribed singularities}

\begin{document}

\begin{abstract}
Motivated by the results of B. Berndtsson, in this memoir we use the new  estimates developed  by  W. He to extend a theorem of the second author on  the existence  of weak $C^{1,1}$ geodesics between two smooth non-degenerate K\"ahler potentials to the case where the metrics on the end points may have singularities on some analytic set and may be positive semi-definite. 

\end{abstract}

\maketitle

\section{introduction}

Towards the end of the 1980's, Mabuchi in \cite{Ma87} and Bourguignon in \cite{ Bo85} independently started studying the geometric structure of the space of k\"ahler metrics in a fixed cohomology class over a given manifold. Later, apparently independently, Semmes in \cite{Se92} and Donaldson in \cite{Do96} rediscovered this idea from a different viewpoint. The fundamental idea behind these works was the introduction of a Weil-Petersson-type metric on this space that would endow it with a riemannian metric and a compatible connection.
Since two K\"ahler forms in a given cohomolgy class differ by the \(dd^c\) of a potential, in order to study the space of metrics, it will suffice to study the real K\"ahler potentials after fixing a background K\"ahler form \(\om\). \footnote{Henceforth, we shall assume that \(\om\) is a fixed smooth positive definite k\"ahlerian metric in the background. Also, the norm of  various tensors are measured with respect to  \(\om\).}
One is lead therefore to studying the space \(\calH\) defined as follows:
\begin{equation*}
\calH
:=
\left\{ \phi \in C^{\infty}(X)| \om + dd^c \phi > 0 \right\}
\end{equation*}

The study  proceeded by formal calculations of the connection coefficients, curvature, and the geodesic equation. 
For example, it was shown that \(\calH\) is a locally homogeneous space of non-positive sectional curvature. 
As in the usual case of riemannian manifolds, one can define the length and energy of a curve, and by taking the first variation of the energy functional for curves, one can derive the geodesic equation.
 It was proved that if the curve \(\phi(t)\) is a geodesic in \(\calH\), it must satisfy the equation \footnote{Two remarks on notation: Here and hereafter, differential operators without subscript are assumed to be in the space direction. In certain occasions, we may use the subscript \(X\) for a space differential operator for emphasis. Also, both \(\del_t\) and \(t\) as subscript are used to denote time derivatives.}
\begin{equation}
\label{ge}
\del_{tt} \phi -{1 \over 2}\vert d \del_t \phi \vert_{\phi}^2 = 0 
\end{equation}

               In \cite{Do96}, Donaldson formulated the relation between the geometry of \(\calH\) and some older problems in few conjectures. This initiated a new programme for studying some classical problems in K\"ahler geometry.
 More specifically, these conjectures related problems such as the existence and the uniqueness of extremal K\"ahler metrics and K\"ahler metrics of constant scalar curvature to the geometry of the space \(\calH\).
  Of particular importance were the geodesics and the the metric structure of the space \(\calH\).

Further, it was proved that equation \ref{ge} can be written as a homogeneous complex Monge-Amp\`ere equation. 
More precisely, let  \(\phi_1,\phi_2 \in \calH\) be two k\"ahlerian potentials on \(X\). 
Then let \(\ta\)  be  the complexification of the time variable \(t \in [0,1]\). We complexify the time variable in the following manner: put \( \ta:= t + \sqrt{-1} \si \), where \(\si \in S^1\) and the boundary data is extended identically along the imaginary direction. Then, \(
\ta\) will belong to \(\Si\), a cylinder of unit height and radius. Assume that \(\pi: X \times \Si \rightarrow X\) is the projection on the first component. We then have the following \( (1,1)\)- forms on the product: \(\Om:= \pi^* \om\) and \(\Om_\phi = \Om + dd^c_X \phi + dd^c_{\tau} \phi \).
 A curve, therefore, can be thought of as a potential \(\phi(x,\ta) \) on \(X \times \Si\). 
 The equation of the geodesic connecting the two potentials \(\phi_1\) and \(\phi_2\) will then be equivalent to the following boundary value problem:

\begin{eqnarray}
\label{HCMA-1}
\begin{cases}
  \Om_{\phi}^{n+1} = 0 \text{  on  }  X \times \Si \nonumber \\
\phi|_{\del \left( X \times S^1  \times \{j\} \right)} = \phi_j \text{ }, j=0,1
\end{cases}
\end{eqnarray}

Donaldson had conjectured that in \(\calH\) the geodesics realise the minimum distance, and he further conjectured that the geodesic connecting two smooth potentials is smooth. The second author proved the following theorem in \cite{Ch00}:

\begin{theorem} \label{ChTh3} \cite{Ch00} Assume that the potentials on the end points, \(\phi_0\) and \(\phi_2\), are smooth and \(\om + d d^c \phi_j > 0, \text{  } j=0,1 \). Then, there exists a  function with bounded \(\del \bdel\)-derivatives on \(X \times \Si\) which solves the geodesic equation weakly. More precisely, there exists a geodesic path \( \phi(t) : [0,1] \rightarrow \calH_{1,1} \) and a uniform constant such that \( 0 \leq dd^c_{X \times \Si} \phi  \leq C \), wherein the subscripts for \(dd^c\) are to denote that the operators are restricted to the corresponding tangent directions.
\end{theorem}

The existence theorem in \cite{Ch00} requires the K\"ahler metrics on the boundaries to be smooth and strictly positive everywhere. In a more recent work presented in \cite{He12}, from where our estimates are inspired, W. He  has proved  that the assumption on the boundary data may be weakened by considering the original geodesic equation, equation \ref{ge}. Namely, one may still prove regularity of the solution for the boundary conditions whose associated metrics are possibly positive semi-definite and have a bound on their laplacians. 
\begin{theorem}
\label{He-1}
\cite{He12} Let \(\phi_0\) and \(\phi_1\) be two potentials with bounded laplacian. Then, there exists a generalised solutions, \(\phi(t)\), of the geodesic equation, such that 
\[
0
\leq
n+\De \phi(t)
\leq
C
\]
where \(C=C\left( \Vert \phi_{0,1} \Vert_{\infty}, \Vert \De \phi_{0,1} \Vert_{\infty}, \om \right) \). \footnote{In this article, the laplacian, \(\De\), is defined so that it has negative spectrum.}
\end{theorem}

    A modification of the estimates used to prove \ref{He-1} will be used in the present note in order  to derive weighted laplacian estimates. Combined with the result in \cite{Be11}, it yields the following \emph{a priori} bound, independent of \(\eps>0\).

\[
| \del_t \phi| + |\Delta \phi| < C
\]

   A  class of singularities are the so-called \emph{conical} singularities a particular case of which are the \emph{orbifold-type} singularities. They were studied on Riemann surfaces  by E. Picard \cite{Pic8}. Recently, after Donaldson's linear theory \cite{Do11} on conical K\"ahler-Einstein metrics, many interesting results have emerged. For instance, cf. \cite{Br11, JMR11,LS-12} and \cite{SW13}. In \cite{Gu12} this is extended to Poncar\'e-type singularities. This generalises the work of Troyanov in \cite{To91} on preassigning the curvature on Riemann surfaces with prescribed conical singularities. Also, in his study of extremal hermitian metrics, the second author considered conical metrics in \cite{Ch99}.

    Existence of geodesics between conical metrics was studied in \cite{CZ12}, where the approaches of \cite{JMR11} and \cite{Do11} were combined with those in \cite{Ch00} in order to obtain the following.

\begin{theorem}
\cite{CZ12}
Let \(\calH_C \subset \calH_{\om_0} \cap C_{\be}^{2,\al}\) denote the K\"ahler potentials with bounded Levi-Civita connection and lower bound on Ricci curvature. Then, any two K\"ahler cone metrics in \(\calH_C\) can be connected by a unique \(C^{1,1}_\be\) cone geodesic. 
\end{theorem}

In particular, this proves the existence of geodesics once the cone angle is small, \(\be < {1 \over 2}\). The methods used to prove the theorem above are more intrinsic than the methods we have adopted here in the sense that analysis is done in appropriate function spaces with respect to the cone metric itself as opposed to our approach, which uses a smooth  reference metric.

\section{Main results}
Our main focus in the present note will be on proving the a theorem which guarantees existence and uniqueness of geodesics between two metrics with conical singularities along a given divisor such that at a given time slice, the space derivatives are bounded away from the divisor. Our method is a modification of the estimates by W. He and is motivated by the work of Berdtsson in \cite{Be11}, where existence of weak geodesics between bounded potentials is used in order to prove uniqueness of weak K\"ahler-Einstein metrics with bounded potentials in case of positive Chern class. In light of the crucial r\^ole this theorem  playes in the recent work of Donaldson, Chen, and Sun, \cite{Chen-Donaldson-Sun}, it is perhaps important to give
a more direct proof of Berdntsson's result on the uniqueness of geodesics between two singular K\"ahler Einstein metrics with $C^{1,1}$ bounded potentials, utilizing the main theorem in this paper.

We prove the following \footnote{For notations and definitions, see Section \ref{not-def}}

\begin{theorem}
\label{Th-1.3}
Let \(\phi_0\) and \(\phi_1 \) be two potentials whose corresponding metrics \(\om_{\phi_0}\) and \(\om_{\phi_1}\) have conic singularities of angle \(\be\) along the smooth divisor \(V\). Then, there is a unique weak geodesic, \(\phi(t)\), connecting them in the following sense: the solution is everywhere H\"older continuous, and on compact sets away from the singularity, its -spatial- complex hessian is uniformly bounded.
\end{theorem}

We can then formulate the more general version of theorem \ref{Th-1.3} the previous theorem as follows:
\begin{theorem}
\label{Th-1.8}
Assume that \(\phi_0, \phi_1\) are two potentials belonging to the space \(\calH_{\xi}(X)\) for a singularity \(\calS\), and an admissible weight function \(\xi\). Then, there is a unique weak geodesic connecting them with bounded laplacian away from the singular set.
\end{theorem}

\begin{proof}[Proof of Theorem \ref{Th-1.3}]
We prove the Theorem \ref{Th-1.3} by considering the following family of boundary value problems to approximate the degenerate equation \ref{ge}.

\begin{equation}
\label{eq-1.5}
\begin{cases}
\left ( \phi_{tt} - \vert \del \phi_t \vert_\phi^2 \right) \omega_\phi^n = \frac{\eps e^f}{\left( \vert s \vert^2 +\eta \right)^p} \om^n \\
\phi(x,i) = \phi(i), \text{  } i = 0, 1. 
\end{cases}
\end{equation}

and its equivalent form 
\begin{eqnarray}
\label{eq-1.6}
\begin{cases}
{  \Om_{\phi}^{n+1} \over \Om^{n+1}} =  \frac{\eps e^f}{\left( \vert s \vert^2 +\eta \right)^p}\text{  on  }  X \times \Si \\
\phi|_{\del \left( X \times R \right)} = \phi_j \text{ }, j=0,1
\end{cases}
\end{eqnarray} \newline

Uniqueness of generalised solutions  amongst bounded potentials is already known, see for example \cite{Be11}. Here and hereafter, we take \(\eta\) to be \(\eta(\eps)\) so that \(\eta \rightarrow 0\) as \(\eps \rightarrow 0\).  We shall, however, drop the explicit dependence. We remark whenever required as to how this dependence may be chosen so that the estimates will hold. In order to prove existence of solutions to this equation, we will need to prove \emph{a priori} estimates for the second derivative of the potential. These estimates are stated in Theorem \ref{Th-1.4}. Since the second derivative will blow up at a certain rate, depending on the cone angle, close to the divisor, we need to prove that the rate of blowing up is bounded close to the divisor.  In particular, we prove that on any compact set not intersecting \(\calS\), the gradient, \( \Vert \nabla \phi \Vert \), and the laplacian, \(\De \phi \), are uniformly bounded. 

More precisely, let \(\phi^k_i\), for \(i=0,1\), be a sequence of smooth potentials that approximate the boundary data in the following sense: on any compact set \(K\) that does not intersect the singularity, we let \(\phi^k_i \rightarrow \phi_i\) in \(C^{1,\m}\) in such a way that \(\De \phi_i^k\) is preserved uniformly bounded. If \(\be > {1 \over 2}\), then keep \( \Vert \nabla \phi_i^k \Vert \) uniformly bounded as well, and if \(\be \leq {1 \over 2}\), that is, the boundary data is merely H\"older continuous of exponent \(2\be\) across the divisor, keep the \(2\be\)-H\"older norm bounded across the singular set. 
Also, choose \(\eps\) to be \({1 \over k}\) and choose \(\eta\) accordingly as it is allowed for the estimates to hold. ‌Thich has been explicitly derived for each estimate. Under these conditions, one also can make the choice so that right hand side of equation \ref{eq-1.5} will tend to zero and \( \left \{ (\eps_j, \eta_j) \right\}_j \rightarrow (0,0)\). The weak solution will be the limit of the smooth solutions thus obtained once we can prove a uniform bound and a uniform modulus of continuity for these solutions.  In Section  \ref{C-0-estimates} we shall derive the \(C^0\) estimates. In Section  \ref{C-1-estimates}, uniform gradient bounds for the case of differentiable boundary data, and a rate of growth for the gradient in the case of smaller angles will be proved. The latter will be used in Section  \ref{C-alpha-estimates} to show \(\de\)-H\"older norm is bounded across the divisor for some small enough \(\de\). 

One will then have a sequence of smooth solutions \(\left \{ \phi^k \right \}_k\) for the the sequence of boundary value problems with a controlled growth of laplacian close to the divisor and controlled gradient or H\"older norm across the divisor, and with a uniform bound on the time derivative. Therefore, one can extract a subsequence \(\left \{ \phi^{k_m} \right \}_m \) that will converge in \(C^\ga\) for some small enough \(\ga\), to the generalised solution \(\phi\). Therefore, in the generalised sense, \(\phi\) will also satisfy the growth conditions on laplacian and will be of class \(C^\ga\). In particular, the convergence will be in \(C^{1, \m}\) on compact sets that do not intersect the singularity.
\end{proof}

\begin{proof}[Proof of Theorem \ref{Th-1.8}]
The uniqueness is already known from \cite{Ch00}. In order to prove \ref{Th-1.8}, we shall need to prove that that the \(\calH_{\xi}\)-norm of the solutions in the following continuity family are uniformly bounded.

Similar to the case of divisorial conical singularities, we are going to prove existence of solutions to the equation \ref{ge} in an appropriate sense. We state the following generalisation of theorem \ref{Th-1.3} for the the following family of boundary value problems, where \(\eta>0\) is chosen depending on \(\eps>0\).
\begin{eqnarray}
\label{eq-1.8}
\begin{cases}
{\omega_\phi^n \over \om^n}
\left ( \phi_{tt} - \vert d \phi_t \vert_\phi^2 \right) 
=
\frac{\eps e^f}{\prod_j \xi_{j,\eta}}
 \\
\phi(x,i) = \phi(i), \text{  } i = 0, 1. 
\end{cases}
\end{eqnarray}

Here, as in the equation \ref{eq-1.5}, \(\eta\) and \(\eps\) are two parameters, but we shall see how \(\eta\) may be chosen depending on \(\eps\) in order for the estimates to hold. The content of the following theorem is the required bounds for this equation.
\end{proof}

\begin{theorem}
\label{Th-1.4}
In the family of boundary value problems \ref{eq-1.5}, assume that the boundary conditions have conical singularities of angles \(\be_j\) along \(V_j\). Then, for any \(\eps>0\), the solution \(\phi_{\eta}\) of \ref{eq-1.5} 

\begin{itemize}
\item if for all \(\be_j\) we have \(\be_j > {1 \over 2}\), then \( \vert \del_t \phi \vert + \vert \phi \vert + \vert \nabla \phi \vert + \xi |\De \phi| < C\),
\item if, for some \(j\) we have \(\be_j \leq {1 \over 2}\), then 
\( \vert \del_t \phi \vert + \vert \phi \vert + \xi |\De \phi| + \Vert \phi \Vert_{C^{0,\de}} < C\) for any \(\de < 2 \be\).
\end{itemize}
In the expressions above, \(C\) only depends on \(\De f\) and the supremum of  \(\vert \phi \vert + \xi |\De \phi|\) on the boundary. 
\end{theorem}

 \begin{remark}
 Note that this theorem does not guarantee that the conical singularity is preserved, but rather, it proves that the growth of the derivatives at any point is not worse than that of the boundary condition, namely the conical case. As a result, away from the singularity of the boundary data we have the usual  bounds on the \(\del \bdel\)-derivatives.
\end{remark}

\begin{remark}
Since \(C\) only depends on the  \(\calH_{\xi}\)-norm of  boundary data, we may choose the boundary condition to be K\"ahler metrics that are semi-definite.
\end{remark}

\begin{remark}
  In the  the theorem above, we could have stated the theorem in the general case, without differentiating between smaller and larger angles. Namely, we could have used the second estimate for the gradient for both larger and smaller angles. 
\end{remark}

\section{Conical Metrics and more general singularities}
\label{not-def}
In this section, basic facts and definitions will be presented. Most of these observations are in some way proved in \cite{JMR11,Gu12}.

  Let $X^n$ be a compact K\"ahler manifold of dimension $n$ and 
When we talk about a metric with a cone of angle $\be$ along a subvariety we mean a metric whose local model is the following metric on \(\bbC^n\) with a cone of angle \(\be\) along the divisor \( [\fz_1=0]\).
\[
\om_{model}
=
{i \over 2}  |\fz_1|^{2\be-2} d \fz_1 \wge d \bfz_1
 +
{i \over 2} \sum_{j=2} d \fz_j \wge d \bfz_j
\]
After an appropriate -singular- change of coordinates, one can see that this model metric indeed represents a  euclidean cone of total angle $\tta = 2 \pi \be$, whose model on \(\bbR^2\) is the following metric:  \(d\tta^2 + \be^2 dr^2\).
By the assumption on the asymptotic behaviour we we mean there exists some coordinate chart in which the zero-th order asymptotic of the metric agrees with the model metric. In other words, there is a constant \(C\), such that 
\[
{1 \over C} \om_{model}
\leq
\om
\leq
C \om_{model}
\]

This asymptotic behaviour of metrics can be translated to the second order asymptotic behaviour of their potentials.

Of particular interest is the case where the conical singularity occurs along a divisor. That is, assume that \(V \subset X\) is a smooth complex hypersurface. Also, let \( (L,h) \) be the line bundle associated to this hypersurface endowed with some hermitian metric \(h\), and let \( s \in \calO(L) \) be the defining section of \(V\). We can assume that there are multiple hypersurfaces. Unless stated otherwise, we assume that all the hypersurfaces are smooth and that they do not intersect. In this case, we may observe that in the proofs we can consider the weight functions and hypersurfaces individually.     Let us set the following notation for the rest of this article. Let \( (L_j, h_j)\) be holomorphic line bundles endowed with hermitian metrics \(h_j\). We then refer by \(s_j\) to some -global- holomorphic section of \(L_j\), an element of \(H^0(X,\calO_X(L_j))\). Also, let \(D:=\cup V_j\).

\begin{lemma}
Let \(s_j\) and \( (L_j,h_j) \) be as before. Then, for sufficiently small \(c\), the following \( (1,1) \)-form
\begin{equation}
\om_{\be}
:=
\om
+
c \sum_j dd^c |s_j|_{h_j}^{2  \be_j}
\end{equation}

defines a k\"ahlerian metric with conical singularities of angle \(\be_j\) along \(V_j\).
\end{lemma}

\begin{proof}
Since the divisors do not intersect, we can consider them individually. We shall therefore drop the subscript in what follows. Adopt a coordinate system in a neighbourhood so that in this coordinate system the divisor corresponds to \([ \fz_1 = 0]\). Also, choose unit vector \(e\) in that neighbourhood for \((L,h)\). Then, we shall have that \( s = \si(\fz) e\) for some holomorphic function, and further, that \(|s|_h=|\si|\). Now, by differentiating in local coordinates, one observes that \(dd^c |s|_h^{2\be}\) can be decomposed into a smooth part and a conical part.
\end{proof}

It can be then observed that if we set a smooth metric $\om$ in the background and if we let  \( \om_\phi = \om +  d\dc \phi \) be a conical metric, then, close to the singular set, the laplacian of the potential with respect to the reference metric \(\om\), \( \Delta \phi \), grows at the rate of \( |s|_h^{ 2\be-2} \). That is, \( \Delta \phi = O(|s|_h^{ 2\be -2 })\). Similarly, we have about the first derivative that \( |\nabla \phi| = O(|s|_h^{2\be - 1}) \).

The advantage of using \( |s|_h\) is that it is a global function and has an intrinsic geometric meaning. One can, however, observe that as long as the hypersurface is smooth, close to the hypersurface, we could substitute \(|s|_h\) with the distance function to the divisor, call it \(\rho_D(x)\). Since the distance function is not smooth farther from the support of the divisor, we can define \(\rho_D\) to be the distance, with respect to the reference metric \(\om\), to the support of the divisor in the vicinity of \( \supp{D}\) and extend it smoothly to the rest of \(X\). This family of distance functions will be used when we consider more general singularities. We can, therefore, state the growth rate of derivatives in terms of \(\rho_D\) as well.

\begin{definition}
\label{admiss}
Assume that \(\calS\), the singular set, is a subset of the manifold \(X\). A function \(\xi\) is called to an \emph{admissible defining function} or \emph{admissible weight function} for the set \(\calS\) if the following conditions are satisfied:
\begin{itemize}
\item the function \(\xi\) is an exhaustion function for \(\calS\), that is it vanishes on,  and only on \(\calS\) with precompact sub-level sets,

\item the complex hessian, with respect to the reference metric \(\om\), of \(\log \xi\) is uniformly bounded from below on \(X - \calS\).
\end{itemize}
\end{definition}

In other words, we require that the mixed derivaives of \(\log \xi\) be currents bounded from below by some multiple \(-C\) of the K\"ahler form. The following observation provides us with two important families of weight functions.

\begin{lemma}
\label{lemma-3.3}
Let \(V \subset X\) be a complex hypersurface.

\begin{sub}[a] Assume that \(\rho_V\) is a function equal to the distance to \(V\) in a tubular neighbourhood of \(V\) and extended smoothly on the rather points. Then, any positive power of the distance function to \(V\), \(\rho_V^\nu\) for some \(\nu>0\), is an admissible weight function.

Further, if \(V_j \subset X\) are hypersurfaces and \(\rho_{V_j}\)'s are the corresponding weight functions, the product \(\prod \rho_{V_j}\) is also an admissible weight function for \(\cup V_j\).
\end{sub}

\begin{sub}[b] 
Assume that \((L,h)\) is a hermitian holomorphic line bundle and \(s \in H^0(X, \calO_X(L))\) is the defining section of \(V\). Then, \(|s|_h^\nu\) with \(\nu>0\) is an admissible weight function.
\end{sub}

\begin{sub}[c]
More generally,  any analytic set admits admissible weight functions. More specifically, assume that the set \(\calS\) is the common zero locus of the holomorphic functions \(f_1, . . ., f_N\). Then, \(\xi\) may be taken to be any power of \(\sum_{j=1}^N \vert f_j(\fz) \vert^2 \) is an admissible weight function.
\end{sub}

\end{lemma}

\begin{proof}
The fact that admissibility of weight functions is preserved under multiplication follows from its definition. The proofs for other claims follow from calculations in local coordinate systems around the submanifolds.
\end{proof}

It will make some of statements clearer if we introduce certain weighted functional spaces. As before, let $\xi$ be a weight function for the gradient and the laplacian of the potential measured with respect to the smooth background metric \(\om\). Since we only use continuous potentials, we have not put a weight on the growth of the \(C^0\). Then, we define the following spaces:

\begin{definition}
Let \(\xi\) be an admissible weight function. We say that a continuous potential \(\phi\) belongs to the space \(\calH_{\xi}\) if  \(\om + dd^c \phi \geq 0\) and it further satisfies \( \Vert \phi \Vert_{\xi}:=|\phi|+ \vert \Delta \phi \vert \xi < \infty \). We may also speak of an admissible pair \((\calS, \xi)\) consisting of the singular set and an admissible weight function.
\end{definition}

\begin{theorem}
\label{Th-1.12}
Assume that in the boundary value problem \ref{eq-1.8}, the boundary data, \(\phi_0\) and \(\phi_1\) belong to the weighted space \(\calH_{\xi}\) for some admissible pair \((\calS,\xi)\). Then, for any \(\eps>0\), the  \(C^{1,1}\) norm of the solution on any compact subset away from \(\calS\) is bounded independent of \(\eta\). More precisely, for any \(\eps>0\), we have the following bound independent of \(\eta\).
\begin{equation}
|\del_t \phi|
+
\xi \vert \De \phi \vert < C
\end{equation}
for any \(\m \in (0,1)\).
\end{theorem}

Since the proof of theorem \ref{Th-1.12} is \emph{mutatis mutandis} the same as that of \ref{Th-1.4}, we shall only prove the latter.
\\

\section{\(L^\infty\) estimate}

\label{C-0-estimates}
There are various ways to see that the solutions of boundary value problem \ref{eq-1.5} are bounded. One can, for example, generalise the argument \cite{Ch00}, where  sub- and super-solutions are constructed, to the case of less regular boundary data.
\begin{proposition}
In the boundary value problem 
\begin{equation}
\begin{cases}
\frac{\Om_\phi^{m}}{\Om^{m}} =   \psi \\
\phi_{|\del(X \times \Si)} = \phi_{0,1}
\end{cases}
\end{equation}

the \(L^\infty\) norm of \(\phi\) in the interior is bounded as soon as the right hand side is square integrable, \(\psi \in L^2(X \times \Si,\Om)\).

\end{proposition}
 
\begin{proof}
Let \(h\) be a solution to the following boundary value problem:
\[
\begin{cases}
\De h = n + 1 \\
\phi_{|\del(X \times \Si)} = \phi_{0,1}
\end{cases}
\]
One can verify that \(h\) is indeed a supersolution. Also, let \(\phi_0\) be a an \(\Om\)-plurisubharmonic function on \(X \times \Si\) whose restriction to the boundary agrees with \(\phi_{0,1}\). Clearly, \(\phi_0\) is a subsolution. Therefore, \(\phi\) is bounded from above and below on \(X \times \Si\).
\end{proof}

\begin{remark}
In order for this estimate to hold, one needs to choose \(\eta(\eps)\) so that the right hand side in boundary value problem \ref{eq-1.5} stays uniformly bounded, namely, \(\eps \leq \eta^p\).
\end{remark}


\section{\(C^\al\) estimates}
\label{C-alpha-estimates}
In case of certain singularities, including the case of conical metrics, we can prove that the \(C^\al\) norm of the solutions are bounded. In case of conical metrics of angles \(\be > {1 \over 2}\), since the potential is \(C^1\), this information will be superfluous. However, in the case of smaller angles, \(\be \leq {1 \over 2}\), we shall prove that not only away from the singularity, but also across the divisor H\"older continuity of the solutions is preserved for some appropriate exponent. 

Let us first make some  observations:
\begin{lemma}
\label{lemma-4.1}
Assume that \(g :[0,1] \rightarrow \bbR\) is continuous on \((0,1]\). Further, assume that \(g\) is locally lipschitzian of constant \(\la(\ta)\) on intervals of the form \([\ta,1]\). Then, \(g\) is H\"older continuous of exponent \(\m\) on \([0,1]\) provided that the following holds:
\[
\varlimsup_{x \rightarrow 0} x^{1 - \m} \la(x) 
<
\infty
\]

\end{lemma}

\begin{proof}
 This can be seen by decomposing the interval \([0,1]\) into subintervals with end points belonging to the sequence \(\left\{ 2^{-n} \right\}_n\).
\end{proof}

The previous lemma will allow us to prove \(\m\)-H\"older continuity in directions transversal to the divisor once we prove the upper bound on the rate of growth of laplacian. We need to prove that the \(\m\)-H\"older modulus is bounded in tangential directions as well. Knowing the rate of growth of the gradient, which is provided in the next section, combined with the following observation, we obtain uniform H\"older continuity.

\begin{lemma}
Assume that \(N^n \subset M^m\) is an immersed submanifold and let \(\rho_N\) be the distance to \(N\). Let \(f \in C^0 \left( M \right) \cap C^1 \left( M - N \right) \) be \(\m\)-H\"older continuous in directions transversal to \(N\). Further, assume that \(\nabla f\), at worst, grows at the rate of \(\rho_N^{-\nu}\) for some \(\nu\). Then, for \(\be={\m \over \m + \nu}\), one has \(f \in C^{\be} \left( M \right) \).
\end{lemma}

\begin{proof}
 Since we already know H\"older continuity in the normal directions, we shall make use of it to prove H\"older continuity in the tangential direction as well. For simplicity, let \(N\) and \(M\) be \(\bbR^n\) and \(\bbR^m\) respectively. 

Since we already assume the control in the normal directions, it will be enough to show that for any two point on the submanifold, \(p,q \in \bbR^n\), \(\left \vert f(q) - f(p) \right \vert \leq C \left \Vert q -p \right \Vert^\al \). More specifically, let \(p=\left( p^1,...,p^n, 0 , ... ,0 \right) , q= \left( q^1, ... ,q^n, 0, ... ,0 \right)\) for simplicity, let us let \( r=\left \Vert q - p \right \Vert\). Choose  \(\ga = {1 \over \m + \nu} \). Also, as usual, let \(e_n\) denote the \(n\)-th element of the standard basis of \(\bbR^m\). Then, by our assumption on the rate of growth of the gradient,

\begin{eqnarray}
\vert f(p) - f(q) \vert 
&\leq&
\vert f(p + r^\ga e_n) - f(p) \vert
+
\vert f(p + r^\ga e_n) - f(q + r^\ga e_n) \vert
+
\vert f(q) - f(q + r^\ga e_n) \vert \nonumber \\
&\leq&
C_1 r^{\ga \m} + C_2 r^{-\ga \nu + 1} 
= C r^{\m \over \m + \nu}
\end{eqnarray}

which proves the claim.
\end{proof}

Now, since our estimate on the rate of growth of laplacian close to the divisor only depends on the \(L^\infty\) bound, cf. section \ref{Laplace-estimates}, we can estimate the rate of growth of the gradient, as is done in the next section. Combined with the preceding lammata, we may obtain the following proposition which can be interpreted as the continuous embedding \(\calH_\xi \hookrightarrow C^\de\).

\begin{proposition}
Let \(\phi\) be a solution of the boundary value problem \ref{eq-1.5} for some \(\eps>0\). Then, for \(\de < 2\be \), \( \Vert \phi \Vert_{0, \de} \leq C\) for some uniform \(C\) which only depends on the boundary conditions and the geometry of the reference metric \(\om\).
\end{proposition}

\begin{remark}
The \(C^\al\) estimate is only useful across the divisor. One can easily observe that away from the divisor the solution is indeed \(C^1\).
\end{remark}

A more straightforward, nevertheless quite restrictive, approach to proving the H\"older estimates is through the \(W^{2,p}\) estimates and the embedding of Sobolev spaces into H\"older spaces which is the content of the following proposition. On complex surfaces, it provides the estimate for \(\be > {1 \over 2}\).

\begin{proposition}
Let the weight function for the laplacian, \(\xi\), satisfy the  property that \( {1 \over \xi} \in L^p\) for  some \(p > d \), where \(d\) denotes the complex dimension of the manifold \(X\). Further, assume that the continuous function \(\phi\) is a solution of boundary value problem \ref{eq-1.5} for some \(\eps>0\). Then, \(\phi \in C^\m\) for \(\m \leq 2 - {2d \over p}\). In particular, in the case of conical singularity along a divisor, it will suffice to have \(\be > 1 - {1 \over d}\).
\end{proposition}

\section{ \(C^1 \) estimate}
\label{C-1-estimates}
We shall derive two different first order estimates, one for the case of differentiable boundary data, corresponding to the cone angle less than half, and the case of larger cone angle.  The distinction, however, is that in the case of smaller cone angle, \(\be > {1 \over 2} \), we prove that the first space derivatives are bounded. 
Note that we shall prove the boundedness of the space gradient. In a general context, the boundedness of the temporal derivative was already proved by Berndtsson:

\begin{proposition}
\cite{Be11}
Let the \(\calH^\infty\) be the set of bounded potentials such that \(\om+dd^c \phi \geq 0\), where the inequality is interpreted in the sense of currents. Assume that \(\phi\) is the solution of the following boundary value problem

\begin{eqnarray}
\begin{cases}
{\omega_\phi^n \over \om^n}
\left ( \phi_{tt} - \vert d \phi_t \vert_\phi^2 \right) 
=
0
 \\
\phi(x,i) = \phi(i) \in \calH^\infty \text{  } i = 0, 1. 
\end{cases}
\end{eqnarray}
 
Then, \(\Vert \del_t \phi \Vert_{L^\infty} \leq C\) for some \(C\) which depends on the geometry of the background metric and the boundary conditions.
\end{proposition}

\begin{proposition}
\label{grad-growth}
In the boundary value problem \ref{eq-1.5}, assume that the boundary conditions have singularities that are no worse that conical singularity of total angle \(2\pi \be\) along the divisor \(D\). Then, if \(\be \geq {1 \over 2}\), we have 
\[
\vert \nabla \phi \vert
\leq 
C
\]

In case of angles strictly larger than \({1 \over 2}\), we have that for any \(\m \in (0,1)\), the gradient of the solution to \ref{eq-1.5} satisfies the following growth condition close to the set \(\supp D\).
\[
\Vert \nabla \phi \Vert \leq C |s|_h^{\m -2 + 2 \be}, \forall \m \in [0,1)
\]
In both cases \(C\) is a constant independent of \(\eps>0\) and only dependent on the boundary conditions and the background geometry.
\end{proposition}

\begin{remark}
Applied to our case, we have an \emph{a priori} growth rate for the laplacian \( \Delta \phi \) which gives a bound for the rate of growth of \(\m\)-H\"older constant of the gradient, \( \nabla \phi \), as we approach the singularity. If we consider the conic singularity, when \( \be < {1 \over 2} \), we have seen in the previous section that \( \nabla \phi \) is bounded. This can be retrieved from the lemma above as well by observing that \(C(t) \lesssim {1 \over \xi} \lesssim O(t^{-2 \be}) \) as \(t \rightarrow 0 \) and therefore, since \( 2 \be < 1\), if we take \( \m \) to satisfy \(0<2\be< \m <1 \), the integral in Inequality \ref{2.10} of Corollary \ref{coro-6-1}  will be finite and therefore \( \nabla \phi \) will be bounded everywhere.
\end{remark}

\begin{remark}If we had an estimate on the -real- hessian tensor of \(\phi\), we could have integrated it to obtain the gradient estimate. However, bounding the growth of laplacian only allows us to bound the growth of \(C^{1, \m}\) norm for any \(\m \in [0,1)\).
\end{remark}

\begin{remark}
Proposition \ref{grad-growth} can be generalised to the case of any admissible pair of singularities, cf. Definition \ref{admiss}. Indeed, we need the more general form in the proof of Theorem \ref{Th-1.12} we need the more general version whose details we have omitted for the sake of simplicity of this exposition.
\end{remark}

Proposition \ref{grad-growth} combined with the  bound obtain in \cite{Be11} for \(\del_t \phi\) gives the following:

\begin{corollary}
\label{coro-6-1}
For the boundary value problem \ref{eq-1.5}, subjected to the same conditions as in proposition \ref{grad-growth}, we have the following \emph{a priori} estimate when \(\be < {1 \over 2}\):

\begin{equation}
\Vert \nabla \phi \Vert |s|_h^{ 2 - 2 \be - \m} + \vert \del_t \phi \vert 
<
C
\end{equation}
where \(C\) is independent of \(\eps\), and \(\m \in [0,1)\).
\end{corollary}

In order to prove the uniform \(C^1\) bounds in the space directions, as in the case of \(C^\al\) bounds, we shall use the rate of growth of laplacian close to the divisor, which, in turn, as we shall see, only depends on the rate of growth of laplacian on the boundary and the uniform \(L^\infty\) estimate. 
Since the laplacian estimates only depend on the \(L^\infty\) estimates of the solution and not the gradient estimates, we can use them to extract information about the rate of growth of the gradient. In particular, in case of a cone singularity along a sub-manifold, we have that the rate of growth of the first derivative close to the divisor is \( O(r^{2\be-1})\) when $\be \leq {1 \over 2}$. Moreover, we shall prove that, provided that \(\be > {1 \over 2}\), the derivative is uniformly bounded as it is the case for the boundary conditions.

In the case where the angle, \(\be\), is smaller than \({1 \over 2}\), however, even the boundary values might not be differentiable, but on the boundary we have a control on the rate of blows-up of the gradient , namely \(|\nabla \phi| = O(r^{2\be-1}) \). Let us consider the sub-level sets of the weight function, \(\{ \xi \leq t\} \). We will give the rate of growth of \(\m\)-H\"older constant, \(C_\m(t)\), of \(\nabla \phi\) on the set \(\{ \xi \leq t\} \). Indeed, since we have bounded the rate of growth of laplacian close to the divisor, we know that for any \(\m\) such that \(0<\m<1\) the \(\m\)-H\"older constant, call it \(C_\m(t)\), has, in the worst case, the rate of growth of the laplacian, \(\De \phi\). That is, \( C_\m(t) \lesssim \xi^{-1} = O(r^{2 \be -2}) \). Hence, using \ref{3.7} we obtain that for any \( 0<\m < 1\), \( |\nabla \phi|  \lesssim O(r^{\m -2+2 \be}) \) as \( r \rightarrow 0 \) which is not as strong as the growth rate on the boundary, \(O(r^{ 2 \be-1}) \).

 Let \(s\) be the defining section of the \emph{smooth} divisor \( V \). This means \( \nabla s \) is no where vanishing. Therefore, there exists some positive number \(\de > 0\), such that \( \Vert \nabla s \Vert > \de > 0 \) along the divisor. We can therefore state the following lemma. We shall omit the proof since the idea is  similar to that of \ref{lemma-4.1}.

\begin{lemma}
\label{3.7}
 Assume $g: [0,1] \rightarrow \bbR$ is continuous on $(0,1]$. Further, assume that \(g\) is of H\"older class for some exponent $\m$ on any set of the form $[\ta,1], 0<\ta$ with constant $C(\ta)$. Then, \(g\) is -uniformly- bounded if  
\begin{equation}
\label{2.10}
\int_0^{1}  t^{\m -1} C(t) d t
<
\infty
\end{equation}

More generally, if the integral above does not converge, the $C^0$-norm on $[\ta,1]$ will not grow worse than the function \(\Theta(\ta)\) defined as follows.
\begin{equation}
\label{2.11}
\Theta(\ta):=
\int_\ta^{1} t ^{\m - 1} C(t) d t
\end{equation}
In particular, if $C(t) = O(t^{-\ga})$, we can re-write the expression in \ref{2.11} as follows:

\[
\Theta_{\m}(\ta)
:=
\int_\ta^{1}  t ^{\m - 1} t^{-\ga} d t = O\left( \ta^{\m - \ga} \right) 
\]
About which we have the following $\Theta (\ta) \lesssim O(\ta^{\m - 2 \be}) \). 
\end{lemma}

\begin{proof}[Proof of Proposition \ref{grad-growth}] We know already that for any \(\eps>0\), the gradient is bounded on uniformly on a compact set, call it \(K\), away from the singularity. We can now use lemma \ref{3.7} and integrate on a curve connecting some point \(q \in K\) to the a point close to the singularity. In the vicinity of the divisor, we can choose this point to be along the shortest path to \(\calS\). We may now apply the last lemma to the rate of growth of the laplacian and obtain the following

\[
\Vert \nabla \phi \Vert
\leq 
C r^{2 \be -2 + \m}
\]

\end{proof}

\section{Laplacian estimates}
\label{Laplace-estimates}
In order to prove the second order estimates, we adopt the approach presented in \cite{He12}.

\subsection{Divisorial singularities}
 As discussed before,  consider the divisor  \(V\) and its defining section \(s\).

 Consider the family of equations 
In this section, we shall prove the laplacian estimate stated in the following proposition:
   
 \begin{proposition}
Let \((L_j,h_j)\) and \(s_j\) be as before. Assume that in the following family of boundary value problems

\begin{equation}
\label{2}
\left ( \phi_{tt} - \vert d \phi_t \vert_\phi^2 \right) \omega_\phi^n = \frac{\eps e^f}{ \left( \vert s  \vert^2 + \eta \right)^ p} \om^n
\end{equation}

the boundary conditions satisfy the condition \( \De \phi_{k} |s|_h^{p} <C< \infty \), for \(k=0,1\). Then, the same holds for \(\eps \in (0,1]\), independent of \(\eps>0\) provided that \(\eps \leq \eta^p\).
\end{proposition}

\begin{proof}
In order to prove this, we consider, for a fixed $\eps$, the family of equations:

\begin{equation}
\label{2}
\left ( \phi_{tt} - \vert d \phi_t \vert_\phi^2 \right) \omega_\phi^n = \frac{\eps e^f}{ \left( \eta + \vert s \vert^2 \right)^p} \om^n
\end{equation}

With the notation as in \cite{He12}, we shall prove that $w_{\eta}:= \zeta_{\eta}^p \left( n+\Delta \phi \right) := \left( \vert s \vert^2 + \eta \right)^p \left( n+ \Delta \phi \right)$ can be estimated.  
The family of functions \(\zeta_{\eta}:=\left( \vert s \vert^2 + \eta \right) \) are approximations of \(\zeta\) by functions and all have positive lower bounds.   We use the fact that for any \(\eps, \eta >0\), when the right hand side in \ref{2} finite, the linearised equation is the laplacian with respect to the metric \(\Om_{\phi}\) on the manifold with boundary \(X \times \Si\),  and therefore satisfies the maximum principle. Namely, for any positive \(\eps\), the linearised operator, which we shall denote by \(\fD\), attains its maximum on the boundary. We use this fact in order to prove that the quantity \( \log w_\eta \) in the interior, that is for the time \(0<t<1\), is controlled by its value on the boundary, \(t=0,1\).

We can rewrite the equation as follows:
\begin{equation}
\label{4.4}
\log \det \left( g_{\al \bbe} + \phi_{\al \bbe} \right)
+
\log \left( \phi_{tt} - \vert d \phi_t \vert_{\phi}^2 \right)
=
\log \eps 
+ 
f
+
\log \det g_{\al \bbe}
-
p \log \zeta_\eta
\end{equation}

Let \(\fD\) denote the linearisation of the left hand side in \ref{4.4}. One can easily verify that

\begin{equation}
\label{linearisation}
\fD \psi = \Delta_\phi \psi 
+
 \frac{\psi_{tt} + g^{\al \bla}_\phi g^{\ka \bbe}_\phi \phi_{t\al} \phi_{t \bbe} \psi_{\ka \bla} - g^{\al \bbe}_\phi \left( \psi_{t\al} \phi_{t \bbe} 
 +
  \psi_{t \bbe} \phi_{t\al} \right)}{\phi_{tt} - \vert d \phi_t \vert_\phi^2}
\end{equation}

In order to keep the expressions shorter, let us choose a shorthand for what we will call the `geodesic operator' as follows:


\begin{equation}
\calG(\phi)
:=
\del_{tt}\phi - {1 \over 2} \vert d \phi_t \vert_\phi^2
\end{equation}

Of course, by the definition of our continuity family, for any \(\eps>0\) we have \(\calG(\phi_{\eps}) > 0 \). We will estimate $\fD \log \zeta_\eta^p = p \fD \log \zeta_\eta$ from below in terms of $\phi$ and $\Delta \phi$.
If we apply \ref{linearisation} to $\zeta_\eta$, for which we of course have $\del_t \zeta_\eta=0$, we shall obtain the following:

\begin{eqnarray}
\label{4}
\fD \log \zeta_\eta 
&=&
 \Delta_\phi \log \zeta_\eta 
 +
  \frac{ g_{\phi}^{\al \bbe} g_\phi^{\ka \bbe} \phi_{ti} \phi_{t\bbe} \left( \log \zeta_\eta \right)_{\ka \bbe}    }{\calG(\phi)} 
\end{eqnarray}

We shall also need the following observation, that in a given coordinate system and at a point off the divisor we have
\begin{equation}
\label{6/1}
{\del^2 \over \del \fz_l \del \bfz_l} \log \left( |s|_h^2 + \eta \right)
\geq
{\vert s \vert_h^2 \over \vert s \vert_h^2 + \eta} 
{\del^2 \over \del \fz_l \del \bfz_l} \log  |s|_h^2
\geq
-C_1
\end{equation}

for some positive constant \(C_1\). To see why the inequalities of \ref{6/1} hold, let us first recall that on \(X-D\), the \( (1,1)\)-form \(d \dc \log |s|_h^2\) represents the first Chern form of the hermitian line bundle \( (L,h)\). Therefore, if we take the trace of this form with respect to the K\"ahler form \(\om_\phi\), we shall have the lower bound for some constant \(C_1\) which depends only on the geometric properties of \((L,h)\). This also shows that the lower bound holds on the entire manifold \(X\) so long as \(\eta>0\). Indeed, one can observe that as  currents the following holds on the entire \(X\):
\[
dd^c \log  \left( |s|_h^2 + \eta \right)
\geq
dd^c \log  |s|_h^2
\geq
-C_1 \om
\]

From this point to the end of this section we will postpone the proof of some inequalities to the appendix, where the calculations of \cite{He12} are presented in further details. Also, we will introduce the two quantities \(\calE_2\) and \(\calA\) in the following inequality which are clarified in the appendix.

We now consider the linearisation of the left hand side of \ref{4.4} operator, \(\fD\), applied to the function \( \log (n + \De \phi) - C\phi\) for some constant \(C\) whose suitable choice will become clear in the estimates. We shall then have the following inequality which is proved in the appendix:

\begin{eqnarray}
\label{eq-4.8}
\fD \left( \log (n + \De \phi) - C\phi \right) 
&\geq& 
\frac{\Delta f - B -S}{n+\De \phi}
-
p{ \De \log \zeta_\eta \over n + \De \phi}
+
\sum_\la \frac{C-B}{1+\phi_{\la \bla}} - (n+1)C \nonumber \\
&+&
 (C-2B) \frac{1}{\calG(\phi)} \sum_\la \frac{\phi_{t\la} \phi_{t \bla}}{ (1+\phi_{\la \bla})^2} 
+
 \frac{ \calE_2}{(n+\Delta \phi) \calG(\phi)}
-
\calA(n+\De \phi)
\end{eqnarray}
where $B$ is an expression in terms of the $\inf R_{\be \bbe \ka \bka}$ and $S$ is the scalar curvature of $g$.

  We, however, need the combination of this and \ref{4} as follows: 
\begin{eqnarray}
\label{eq-4.9-}
\fD \left( \log w - C\phi \right) 
&=& 
\fD \left( \log \psi + p \log \zeta -C\phi \right)  \geq \frac{\Delta f - B -R}{n + \De \phi} 
-
p{ \De \log \zeta_\eta \over n + \De \phi}
- 
\sum_\la \frac{C-B}{1+\phi_{\la \bla}} \nonumber \\
& -&
 (n+1)C + (C-2B) \frac{1}{\calG} \sum_\la \frac{\phi_{t\la} \phi_{t \bla}}{ (1+\phi_{\la \bla})^2} 
+
 \calE_2 \frac{1}{(n+\Delta \phi) \calG(\phi)}  \nonumber \\ 
 &-&
 \calA(n+\De \phi)
 + 
 p \De_ \phi \log \zeta_\eta + p\frac{1}{\calG(\phi)} \sum_\la \frac{  \phi_{t\la} \phi_{t\bla} \left( \log \zeta \right)_{l \bl}} {\left( 1 + \phi_{l \bl}\right)^2} \nonumber \\
&=& 
\frac{\Delta f - B -R}{h} 
-
 (n+1)C  + \calE_2 \frac{1}{(n+\Delta \phi) \calG(\phi)} 
-
\calA(n+\De \phi)  \nonumber \\
&+ &
(C-B ) \sum_\ka \frac{1}{1+\phi_{\ka \bka}} + 
\frac{1}{\calG} \sum_\la \frac{  \phi_{t\la} \phi_{t\bla} } {\left( 1 + \phi_{l \bl}\right)^2} \left( p \left( \log \zeta \right)_{l \bl} + C - 2B \right)  \nonumber \\
&+&
p \sum_\la \left( {1 \over 1 + \phi_{\la \bla}} - {1 \over n + \De \phi} \right) \left( \log \zeta_\eta \right)_{\la \bla}
\end{eqnarray}

 We show that 
\begin{equation}
\label{eq-4.10}
\frac{1}{\calG(\phi)} \sum_\la \frac{  \phi_{t\la} \phi_{t\bla} } {\left( 1 + \phi_{\la \bla}\right)^2} \left( p \left( \log \zeta \right)_{\la \bla} + C - 2B \right) \geq 0
\end{equation}
for large enough \(C\). To see this, we first observe that  for a column vector \(\al\), the hermitian matrix obtained by \( \al \al^{\dagger}\) is non-negative. In particular, the vector can be taken to be \(\al_j=\phi_{tj}\). Also, from \ref{6/1} we know a lower bound for the mixed derivatives of \(\log \zeta \). We can estimate the last term in the equation \ref{eq-4.9} as follows

\begin{equation}
\sum_\la \left( {1 \over 1 + \phi_{\la \bla}} - {1 \over n + \De \phi} \right) \left( \log \zeta_\eta \right)_{\la \bla}
\geq
-C_1 \sum_\la {1 \over 1 + \phi_{\la \bla}}
\end{equation}
where \(C_1\) is the constant from \ref{6/1}. We can then obtain the following:

\begin{eqnarray}
\label{eq-4.9}
\fD \left( \log w - C\phi \right) 
&=& 
\frac{\Delta f - B -R}{h} 
-
 (n+1)C  + \calE_2 \frac{1}{(n+\Delta \phi) \calG(\phi)} 
-
\calA(n+\De \phi)  \nonumber \\
&+ &
(C-B - p C_1 ) \sum_\ka \frac{1}{1+\phi_{\ka \bka}}
\end{eqnarray}

In the expressions above, if we let $C$ be a large number, the coefficient of the last term, $C-B-p C_1$, will be larger than 1.

Therefore, \ref{eq-4.8} can gives the following:

\begin{eqnarray}
\label{eq-4.9/1}
\fD \left( \log w - C\phi \right) &\geq& 
\frac{\Delta f - B -R}{n+\De \phi} 
 \nonumber \\
&-& 
 (n+1)C  + \calE_2 \frac{1}{(n+\Delta \phi) \calG(\phi)}  \nonumber \\ 
&-&\calA(n+\De \phi) 
+ 
\sum_\ka \frac{1}{1+\phi_{\ka \bka}}
\end{eqnarray}

After some further manipulation of \ref{eq-4.9}, the details of which can be found in the appendix, we obtain:
\begin{eqnarray}
\label{eq-4.12}
\fD \left( \log w - C \phi \right) \geq \sum_\la \frac{1}{1+\phi_{\la \bla}} - (n+1)C
\end{eqnarray}

Let us observe that, similar to 2.19 in \cite{Ya78}, we have
\begin{equation}
\sum_{\la} \frac{1}{1+\phi_{\la \bla}} 
+
\frac{1}{\calG(\phi)} 
\geq
\left\{ \frac{\sum (1+ \phi_{\la \bla}) + \calG(\phi)}{\left( \prod (1+\phi_{\la \bla}) \right) \calG(\phi)} \right\}^\frac{1}{n}
\end{equation}

Combining  \ref{eq-4.12}, \ref{lin-t-2}, and the preceding inequality we have the following:

\begin{eqnarray}
\fD \left( \log w - C \phi + {t^2 \over 2}\right) 
&\geq& 
\sum_{\la} \frac{1}{1+\phi_{\la \bla}} - (n+1)C + \frac{1}{\calG(\phi)} 
\geq 
\left\{ \frac{\sum (1+ \phi_{\la \bla}) + \calG(\phi)}{\left( \prod (1+\phi_{\la \bla}) \right) \calG(\phi)} \right\}^\frac{1}{n} \nonumber \\
&-&
(n+1)C \nonumber \\
&=& \left( n + \Delta \phi + \calG(\phi) \right)^\frac{1}{n} \left( \frac{ (|s|^2+\eta)^p}{\eps e^f} \right)^\frac{1}{n} - (n+1)C \nonumber \\
&=& \left \{ (n + \Delta \phi )(|s|^2+\eta)^p +(|s|^2+\eta)^p \calG(\phi) \right \}^\frac{1}{n} \left( \eps s^f \right)^\frac{-1}{n} - (n+1)C \nonumber \\
&\geq&  w ^\frac{1}{n} \left( \eps s^f \right)^\frac{-1}{n} - (n+1)C
\end{eqnarray}

Having this differential inequality, one can argue that either $\log w - C\phi + t^2$ attains its maximum at some interior point $P$, in which case $\fD \left( \log w - C\phi + t^2 \right) \leq 0$, which gives the following upper bound for $w$:
$$w(P)  \leq \eps e^f \left( (n+1)C \right)^n$$
or the maximum of $\log w - C\phi + t^2$ occurs on the boundary.  
\end{proof}

The calculations are valid for arbitrary \(p\). In the case of conical singularity, we have seen that \( \left( \De \phi \right) |s|_h^{2\be} \) is bounded on the boundary, therefore, we can choose \(p = \be\). Of course the proposition will still hold for larger \(p\) as well, however, that will not be optimal.

\subsection{Non-divisorial singularities}
We may observe that in the proof of Theorem \ref{Th-1.4}, the only property of \(|s|_h^{2\be}\) we have used is boundedness from below of  the mixed derivative of \(\log |s|_h \) on \(X-D\). We can therefore generalise this result to the case of more general singular sets so long as we can find admissible weight functions. 

\begin{proposition}
Let \(W\) be a smooth embedded complex submanifold of \(X\). Then, the distance function to \(W\), \(\rho_W\), is an admissible weight function.
\end{proposition}

Similar to what we did in section zero, we consider the family of equations modified as follows:

\begin{equation}
\label{12}
\left ( \phi_{tt} - \vert d \phi_t \vert_\phi^2 \right) \omega_\phi^n = \frac{\eps e^f}{ \xi_\eta} \om^n
\end{equation}

where \(
\xi_\eta
=
\left( \rho(\fz)^2 + \eta \right)^\al
\)

We claim that the quantity \( 
\xi_\eta \left( \rho(\fz) \right) \Delta \phi(t)
\) stays bounded independent of $\eta$. We shall henceforth denote \( \xi_\eta \) by \(\xi\). \newline

In the end, we can consider the equation modified as follows:
\begin{equation}
\left ( \phi_{tt} - \vert d \phi_t \vert_\phi^2 \right) \omega_\phi^n 
=
\frac{\eps e^f}{\prod_j \xi_{j}} \om^n
\end{equation}
where each $\xi_j$ is a weight function with certain properties that vanishes on a set containing the singularity. Since we only need the mixed derivatives of $\log \xi$, we can merely assume $\xi$ is a function whose $\log$ is $\tta$-plurisubharmonic  for some fixed form $\tta$. This need to hold only in the vicinity of the singular set $\calS$ such that on the singularity we have $\xi |_\calS=0$. We observe that this holds when $\xi = |s|_h$, where $s \in H^0(L, \calO)$ and $(L,h)$ is a holomorphic line bundle equipped with a hermitian metric $h$ by inequality \ref{6/1}. It also holds if we take $\xi$ to be the distance to a complex submanifold containing the singularity. (cf. lemma \ref{lemma-3.3}). 

We also need the following observation.
\begin{lemma}
If \(\xi\) is an admissible weight function, then, the elements of the family of function \(\xi_\eta:=\xi+\eta\), which approximate \(\xi\) by strictly positive functions, have a uniform lower bound on \( dd^c  \log \xi_\eta\), namely, as  currents \( dd^c  \log \xi_\eta \geq -C \om\) where \(C\) is a uniform constant.
\end{lemma}

We can repeat 

\begin{equation}
\fD \left( \log \left( \xi_\eta \phi \right) - C \phi + t^2 \right) 
\geq
\left( \xi_\eta ( n + \Delta \phi) \right) ^\frac{1}{n} \left( \eps s^f \right)^\frac{-1}{n} 
-
(n+1)C
\end{equation}

And it is thus proved that 
\begin{equation}
\sup_{M \times [0,1]} \left\{ \xi (n+ \Delta \phi) \right \} 
\leq
\sup_M  \left\{ \xi (n+ \Delta \phi_{0,1}) \right\} 
\end{equation}
\qed \newline

\section{Some remarks and some special cases of singularities}

We finish this note by some remarks.

\begin{remark}
So far we have only considered the case where the singular set is given by the zero locus of some holomorphic section. 
But thanks to the local nature of the operations, one can merely require that the sigular set be the locally the zero set of a finite number of holomorphic functions. 
In that case also one may take any power of the modulus those local defining  functions to be the `local'	 weight function.  More specifically, 
let \(V\) be the common zero set of function \(f_j, k=1,...,k\).
Then, the function \(\left( \sum_j |f_j|^2 \right)^p\), for \(p>0\), 
is an admissible weight function for the common zero locus of the functions \(\{f_j\}_j\). \\
Further, if the defining functions are defined locally, as in an algebraic variety, one can still construct an admissible  weight  function for \(V\) as follows. 
Let us observe that if one has a partition of unity \(\m_j(x)\) subordinate to \(U_j\), and if one has admissible weight functions \(\xi_j\) on each of the open sets \(U_j\), 
then the function \(\xi:= \sum_j \m_j \xi_j\) is a global admissible weight function. 
This allows us, in particular, to construct an admissible weight function when the singular set is contained in the common zero set of locally defined holomorphic functions.
\end{remark}

\begin{remark}
Having obtained an upper bound for the space laplacian of the potential, \(\De \phi\), we can show that the diameter is uniformly bounded. To see this, let us note that the set \(X - D\) is path connected. Let \(x\) be a fixed point outside of the divisor, \(x \in X - D\). Then, any point \(q \in X-D\) may be connected by a curve \(\ga \subset X - D\). Also, since the divisor is smooth, for any point on the divisor, \(p \in D\), there is a curve \(\ga\) connecting \(p\) to the point \(x\), contained in  \(X - D\) except at \(p\), which is perpendicular to \(D\) at \(p\). 

   Let \(d_{\phi}\left(p, q\right)\) denote the distance with respect to the metric \(\om_\phi\) between two points \(p\) and \(q\). Fix a point \(x \in X - \calS\). Then, by the triangle inequality,
\[
\operatorname{diam} \left( M \right)
\leq
\sup_{p,q} \left( d_{\phi}\left( p, x \right)
+
d_{\phi} \left( x,q \right)
\right)
\]
However, \(d_{\phi} \left( x,p \right)\) can be estimated from above by measuring the length of the curve \(\ga\), connecting \(p\) to \(x\), described in the previous paragraph. The length of any such curve, in turn, can be estimated since we have growth rate of \(O(\vert s \vert_h^{2\be - 2})\) close to the divisor for the metric.

\end{remark}






\begin{remark}\textbf{Singularities along a totally real submanifold}
One may observe that the one example of of admissible function is the distance function to a totally real submanifold, \(\calR\) of \(X^n\) once we one has that, \(n\), the complex dimension of \(X\), is larger than 1.

\end{remark}

\appendix

\section{\label{appA}}
In this appendix, we shall provide the details we omitted, including the proof of \ref{eq-4.8}, in the proof of laplacian estimates. In these calculations we follow \cite{He12}.
 We will henceforth use the normal coordinates, in which at a the given point \(q\), \(g_{\al \bbe} = \de_{\al \be}\), \(\del_{\ka}g_{\al \bbe} = \del_{\bla}g_{\al \bbe}=0\), and \(g_{\phi, \al \bbe} = \de_{\al \be} ( 1 + \phi_{\al \bbe}) \), whenever coordinates appear.

We first recall the equation \ref{linearisation}, the linearisation of the operator:

\begin{equation}
\label{lin-psi}
\fD \psi = \Delta_\phi \psi + \frac{\psi_{tt} + g^{\al \bla}_\phi g^{\ka \bbe}_\phi \phi_{t\al} \phi_{t \bbe} \psi_{\ka \bla} - g^{\al \bbe}_\phi \left( \psi_{t\al} \phi_{t \bbe} + \psi_{t \bbe} \phi_{t\al} \right)}{\calG(\phi)}
\end{equation}

If we substitute \(\psi = \phi\), we obtain the following:

\begin{equation}
\label{lin-phi}
\fD \phi 
=
(n+1) 
-
\sum_\be {1 \over 1 + \phi_{\be \bbe}}
-
{1 \over \calG(\phi)}
\sum_\be \frac{ \phi_{t\be} \phi_{t \bbe}}{ (1 + \phi_{\be \bbe})^2}
\end{equation}

And for \(\psi = t^2\):

\begin{equation}
\label{lin-t-2}
\fD t^2
=
{2 \over \calG(\phi)}
\end{equation}

We shall also need the following identity later in calculations:

\begin{equation}
\label{lin-log}
\fD \log \psi
=
{\fD \psi \over \psi}
-
{g_{\phi}^{\ka \bla} \psi_{\ka} \psi_{\bla} \over \psi^2}
-
\frac{\left( \psi_t - g_{\phi}^{\al \bla} \phi_{t\al} \psi_{\bla} \right) \left(\psi_t - g_{\phi}^{\ka \bla} \phi_{t \bbe} \psi_{\ka} \right) }
{\psi^2 \calG(\phi)}
=:
{\fD \psi \over \psi}
-
{g_{\phi}^{\ka \bla} \psi_{\ka} \psi_{\bla} \over \psi^2}
-
\calA(\psi)
\end{equation}

We have implicitly defined \(\calA(\psi)\) in the identity above.

 Let us substitute \(\psi = \De \phi\) and obtain the following:

\begin{equation}
\label{A-2}
\fD ( \Delta \phi ) 
=
\Delta_{\phi} \Delta \phi
+
\frac{ \Delta \del_{tt} \phi + g^{\al \bla}_{\phi} g^{\ka \bbe}_{\phi} \left( \Delta \phi \right)_{\ka \bla} \phi_{t\al} \phi_{t \bbe} - g^{\al \bbe}_{\phi} \left( \left( \Delta \phi \right)_{t\al} \phi_{t \bbe} + \left( \Delta \phi \right)_{t \bbe} \phi_{t\al} \right)}{ \calG(\phi)}
\end{equation}

We can substitute the first two terms,  namely \( \De_{\phi} \De \phi + {\De \del{tt} \phi \over \calG(\phi)}\), following the calculations in section 2 , equations 2.7 and 2.9, of \cite{Ya78}, and obtain:

\begin{eqnarray}
\label{A-3}
\De_{\phi} \De \phi  
&=&
\sum_\ka g^{\al \bbe}_{\phi} g^{\m \bn}_{\phi} \phi_{\al \bn \ka} \phi_{\m \bbe \bka}
+ 
\De f
-
\De \log  \zeta_\eta^p
-
\De \log \left( \del_{tt} \phi- |d \del_t \phi|_{\phi}^2 \right)
+
I \nonumber \\
&\geq&
\sum_k g^{\al \bbe}_{\phi} g^{\m \bn}_{\phi} \phi_{\al \bn \ka} \phi_{\m \bbe \bka}
+ 
\De f
-
\frac{\De  \left( \del_{tt} \phi - |d \del_t \phi|_{\phi}^2 \right)}
{ \calG(\phi)}
+
\frac{\left \vert  d\calG(\phi)  \right \vert^2}
{\calG(\phi)^2}
+
I
-
\De \log  \zeta_\eta^p \nonumber\\
&\geq&
\sum_\ka g^{\al \bbe}_{\phi} g^{\m \bn}_{\phi} \phi_{\al \bn \ka} \phi_{\m \bbe \bka}
+ 
\De f
-
\frac{\De  \left( \del_{tt} \phi - |d \del_t \phi|_{\phi}^2 \right)}
{ \calG(\phi)}
+
I
-
\De \log  \zeta_\eta^p \nonumber
\end{eqnarray}

where \(C_1\) is the same constant that appears in equation \ref{6/1},  \(I = \sum_{\al,\ka} {1+\phi_{\ka \bka} \over 1 + \phi_{\al \bal}} R_{\al \bal \ka \bka} - S \). 
Here \(S\) and \(R\) denote scalar curvature and  curvature tensor respectively. 
   We note that if we let \(B\) be a positive constant such that \(- B \leq \inf R_{\al \bal \ka \bka} \), then \(I\) satisfies the following inequality:

\begin{equation}
\label{A-7}
I
\geq
- B (n+\De \phi) \sum_\al {1 \over 1 + \phi_{\al \bal}}
-B -S =:
- B (n+\De \phi) \sum_\al {1 \over 1 + \phi_{\al \bal}}
-C_2
\end{equation}

We would like to bound the terms containing time derivatives from below. 

By substituting \ref{A-3} into \ref{A-2}, this leads us to the following:

\begin{eqnarray}
\label{A-8}
  \fD ( \De \phi)
&\geq &
\De f
+
\sum_{\ka} g^{\al \bbe}_{\phi} g^{\m \bn}_{\phi} \phi_{\al \bm \ka} \phi_{\m \bbe \bka}
+
I
+
\calE(\phi) \nonumber - p\De \log \zeta_\eta\\
&+&
\frac{  g^{\al \bla}_{\phi} g^{\ka \bbe}_{\phi} \left( \Delta \phi \right)_{\ka \bla} \phi_{t \al} \phi_{t \bbe} - g^{\al \bbe}_{\phi} \left( \left( \Delta \phi \right)_{t\al} \phi_{t \bbe} + \left( \Delta \phi \right)_{t \bbe} \phi_{t\al} \right)}{ \calG(\phi)}
\end{eqnarray}

where

\begin{equation}
\calE(\phi)
=
{\De \left \vert d \phi_t\right \vert_{\phi}^2 \over \calG(\phi)}
\end{equation}

We now study the term \( \calE(\phi)\)  as follows. 
After calculations in normal coordinates, we observe that in the expression above, the numerator, \(\De |d \phi_t|_{\phi}^2\), can be written as follows

\[
\calE_1+\calE_2 
-
 g^{\al \bla}_{\phi} g^{\ka \bbe}_{\phi} \left( \Delta \phi \right)_{\ka \bla} \phi_{t\al} \phi_{t \bbe}
  +
   g^{\al \bbe}_{\phi} \left( \left( \Delta \phi \right)_{t\al} \phi_{t \bbe} 
 +
 \left( \Delta \phi \right)_{t \bbe} \phi_{t\al} \right)
\]

where the terms are defined as follows:
\begin{eqnarray}
\begin{split}
\calE_1
:=
 &g^{\al \bn}_{\phi} g^{\m \bbe}_{\phi} R_{\m \bn \ka \bla} \phi_{t\al} \phi_{t \bbe} 
 \left( g^{\ka \bla} + \phi_{\ka \bla} \right) \nonumber \\
\calE_2
:=              
&g^{\ka \bla}  
   g_{\phi}^{\al \bn} g_{\phi}^{\m \bbe} 
     \left \{
 	g_{\phi}^{\rho \bsi} \phi_{t \rho} \phi_{t \bbe}
 	\left(
 		\phi_{\al \bsi \ka} \phi_{\m \bn \bla} + \phi_{\m \bn \ka} \phi_{\al \bsi \bla}
 	\right)        
 	+
 	 \phi_{\m \bn \bla} 
 	 \left(
 	 	\phi_{t\al \ka} \phi_{t\bbe} + \phi_{t\al} \phi_{t\ka \bbe}
 	 \right) \right. \\
 	 &\left.+
 	 \phi_{\m \bn \ka}
 	 \left(
 	 	\phi_{t \al \bla} \phi_{t\bbe}+\phi_{t\al} \phi_{t\bbe \bla}
        	\right)
     \right \}
    \\
&+
g^{\ka \bla} g_{\phi}^{\al \bbe} \left( \phi_{t \al \bla}\phi_{t \bbe \ka} + \phi_{t \al \ka} \phi_{t \bbe \bla} \right) \nonumber\\
\end{split}
  \end{eqnarray}

The last calculation allows us to cancel the fourth order terms in \ref{A-8} with those of \(\calE \left( \phi \right) \). Since the derivation of the preceding inequality is done by straightforward, nevertheless long, calculations in normal coordinates, we omit the calculation and refer the reader to 2.10 in \cite{He12}.

We now use \ref{A-8} to obtain

\begin{eqnarray}
  \fD ( \De \phi)
&\geq &
\De f
+
\sum_\ka g^{\al \bbe}_{\phi} g^{\m \bn}_{\phi} \phi_{\al \bn \ka} \phi_{\m \bbe \bka}
+
I
+
{\calE_1 + \calE_2 \over \calG(\phi)}
-
p\De \log \zeta_\eta
\end{eqnarray}

By the definition of \(B\), we have \(R_{p \bq k \bl} \geq -B(\de_{pq} \de_{kl}+\de_{pl}\de_{kq})\). We may now combine this piece of information with the inequality \(n+\De \phi \geq 1 + \phi_{j \bj} \), and obtain:

\begin{eqnarray}
\label{A-10}
\calE_1
&\geq&
- g^{\al \bn}_{\phi} g^{\m \bbe}_{\phi}  B(\de_{\m \nu} \de_{\ka \la}+\de_{\m \la }\de_{\ka \nu}) \phi_{t\al} \phi_{t \bbe} 
  \left( g^{\ka \bla} + \phi_{\ka \bla} \right) \nonumber\\
  &>&
  -2 B(n+\De \phi) \sum_\be \frac{\phi_{t \be} \phi_{t \bbe}}{ (1 + \phi_{\be \bbe})^2}
\end{eqnarray}

Note that \(\calE_3\) appears in the numerator of the last term in \ref{A-8}. We can, therefore, combine \ref{A-7}, \ref{A-8} and \ref{A-10} and get:
\begin{eqnarray}
\fD ( \De \phi)
&\geq &
\De f
+
\sum_k g^{i \bj}_{\phi} g^{p \bq}_{\phi} \phi_{i \bq k} \phi_{p \bj \bk}
-
 B (n+\De \phi) \sum_i {1 \over 1 + \phi_{i \bi}}
-
C_2
+
{\calE_1 + \calE_2 \over \calG(\phi)}
-
p\De \log \zeta_\eta \nonumber \\
&\geq&
\De f
+
\sum_{\ka} g^{\al \bbe}_{\phi} g^{\m \bn}_{\phi} \phi_{\al \bn \ka} \phi_{\m \bbe \bka}
-
 B (n+\De \phi) \sum_\al {1 \over 1 + \phi_{\al \bal}}
-
C_2    \nonumber \\
&-&
{  2 B(n+\De \phi)  \over \calG(\phi)} \sum_\be \frac{\phi_{t \be} \phi_{t \bbe}}{ (1 + \phi_{\be \bbe})^2}
+
{\calE_2 \over \calG(\phi)}
-
p\De \log \zeta_\eta \nonumber \\
&\geq& 
\De f - 
C_2
-
 B (n+\De \phi) \sum_\al {1 \over 1 + \phi_{\al \bal}} 
 +
 {g_{\phi}^{\ka \bla} \left(\De \phi \right)_\ka \left(\De \phi \right)_{\bla} \over n + \De \phi} 
 -
{  2 B(n+\De \phi)  \over \calG(\phi)} \sum_\be \frac{\phi_{t\be} \phi_{t \bbe}}{ (1 + \phi_{\be \bbe})^2} \nonumber \\
&-&
p\De \log \zeta_\eta \nonumber
\end{eqnarray}

Where, in the last inequality, we have used the following consequence of the Schwarz inequality for the third order terms (cf. 2.15 in \cite{Ya78}):
\begin{equation*}
\label{eq-Schw}
\sum_\ka g^{\al \bbe}_{\phi} g^{\m \bn}_{\phi} \phi_{\al \bn \ka} \phi_{\m \bbe \bka}
\geq
{( d \De \phi, d \De \phi)_{\phi} \over n + \De \phi}
=
{g_{\phi}^{\ka \bla} \left(\De \phi \right)_\ka \left(\De \phi \right)_{\bla} \over n + \De \phi}
\end{equation*}

We now use the last inequality, \ref{lin-log}, with \(\psi = n + \De \phi\), and \ref{lin-phi} to obtain the following

\begin{eqnarray}
\fD \left( \log( n + \De \phi) - C \phi \right)
&=&
{\fD \left( \De \phi \right) \over n + \De \phi}
+
{\sum_\ka g^{\al \bbe}_{\phi} g^{\m \bn}_{\phi} \phi_{\al \bn \ka} \phi_{\m \bbe \bka} \over n + \De \phi}
-
\calA( n + \De \phi) \nonumber \\
&-&
(n+1)C
+
C \sum_\be {1 \over 1 + \phi_{\be \bbe}}
+
{C \over \calG(\phi)} \sum_\be {\phi_{t\be} \phi_{t \bbe} \over (1 + \phi_{\be \bbe})^2} \nonumber \\
&\geq&
{\De f -C_2 \over n + \De \phi}
-
p{ \De \log \zeta_\eta \over n + \De \phi}
-
{\calE_2(\phi) \over \calG(\phi) (n+\De \phi)} 
-
\calA( n + \De \phi) \nonumber \\
&+&
 \sum_\la \frac{ ( \Delta \phi )_\la ( \Delta \phi )_{\bla} }{ (1+\phi_{\la \bla}) (n+\Delta \phi)^2} - 
 \sum_\la \frac{\psi_\la \psi_{\bla}}{\psi^2 (1+\phi_{\la \bla})} \nonumber\\
&-&
(n+1)C
+
(C - B) \sum_\be {1 \over 1 + \phi_{\be \bbe}}
+
{C - 2B  \over \calG(\phi)} \sum_\be {\phi_{t\be} \phi_{t \bbe} \over (1 + \phi_{\be \bbe})^2} \nonumber \\
&=&
{\De f -C_2 \over n + \De \phi}
-
p{ \De \log \zeta_\eta \over n + \De \phi}
-
{\calE_2 \over \calG(\phi) (n+\De \phi)} 
-
\calA( n + \De \phi) \nonumber \\
&-&
(n+1)C
+
(C - B) \sum_\be {1 \over 1 + \phi_{\be \bbe}}
+
{C - 2B  \over \calG(\phi)} \sum_\be {\phi_{t\be} \phi_{t \bbe} \over (1 + \phi_{\be \bbe})^2} \nonumber
\end{eqnarray}

for  any constant \(C\), which is \ref{eq-4.8}. \\ \\

We now turn to proving \ref{eq-4.12} based on \ref{eq-4.9/1}. Recall that by \ref{eq-4.9/1} we had:

\begin{eqnarray}
\fD \left( \log w - C\phi \right) 
&\geq& 
\frac{\Delta f - C_2}{n+\De \phi} 
- 
 \calA(n+\De \phi) + \calE_2 \frac{1}{(n+\Delta \phi) \calG(\phi)} 
-
(n+1)C
+ 
\sum_\ka \frac{1}{1+\phi_{\ka \bka}} \nonumber \\
\end{eqnarray}

It will suffice to prove that 
\begin{equation}
\calE_2
\geq
(n+\De \phi)
\calA(n+\De \phi)
 \calG(\phi)
\end{equation}

which is equivalent to the following:

\[
\calE_2
(n+\De \phi)
\geq
\left( \psi_t - g_{\phi}^{\al \bla} \phi_{t\al} \psi_{\bla} \right) \left(\psi_t - g_{\phi}^{\ka \bbe} \phi_{t \bbe} \psi_{\ka} \right)
\]

for \(\psi = n + \De \phi\). Since this also follows from a straightforward calculation, we refer the reader to (2.21) in \cite{He12}.

\end{document}